\documentclass[a4paper,11pt]{article}
 \usepackage[plainpages=false]{hyperref}
 \usepackage{amsfonts,amsmath,amssymb,amsthm}
 \usepackage{latexsym,lscape,rawfonts,mathrsfs}

 \usepackage[dvips]{color}
 \usepackage{multicol}

 \usepackage[all]{xy}
 \usepackage{eufrak}
 \usepackage{makeidx}
 \usepackage{graphicx,psfrag}

 \usepackage{array,tabularx}

 \usepackage{setspace}

 \usepackage{appendix}

\usepackage{pdfsync}

\usepackage{txfonts}

 \newcommand{\D}[2]{\ensuremath{ \frac{\partial{#1}}{\partial{#2}}}}

 \newcommand{\R}{\ensuremath{\mathbb{R}}}
 \newcommand{\Z}{\ensuremath{\mathbb{Z}}}

 \newcommand{\ba}{\begin{align*}}
 \newcommand{\ea}{\end{align*}}

 \DeclareMathOperator{\Vol}{Vol}
 \DeclareMathOperator{\diam}{diam}

 \newcommand{\norm}[2]{{ \ensuremath{\|} #1 \ensuremath{\|}}_{#2}}
 \newcommand{\snorm}[2]{{ \ensuremath{\left |} #1 \ensuremath{\right |}}_{#2}}

 \makeatletter
 \def\ExtendSymbol#1#2#3#4#5{\ext@arrow 0099{\arrowfill@#1#2#3}{#4}{#5}}
 
 \makeatother

 \makeatletter
 \def\ExtendSymbol#1#2#3#4#5{\ext@arrow 0099{\arrowfill@#1#2#3}{#4}{#5}}
 
 \makeatother

 \definecolor{hao}{rgb}{1,0.5,0}
 \definecolor{miao}{cmyk}{0.5,0,0.2,0.2}
 \definecolor{qiao}{gray}{0.96}

 \newtheorem{claim}{Claim}

 \newtheorem{proposition}{Proposition}[section]
 \newtheorem{lemma}{Lemma}[section]

 \newtheorem{remark}{Remark}[section]

 \newtheorem{theoremin}{Theorem}
 
 \newtheorem{corollaryin}{Corollary}
 
 \newtheorem{remarkin}{Remark}

 \title{On the conditions to extend Ricci flow(III)}
 \author{Xiuxiong Chen\footnote{Supported by NSF grant DMS-0907778 and funds from Simons Center for Geometry and Physics.}\;,  Bing Wang\footnote{Supported by NSF grant DMS-1006518.}}
 \date{}

\begin{document}
 \maketitle


\section{Introduction}

 In this paper, we develop some estimates for the Ricci flow. Some of them are the improved versions of the estimates in~\cite{BWa2}, some of them
 are purely new. These estimates are useful in the study of Ricci flow with bounded scalar curvature.
 In particular, they have applications in the extension problem of Ricci flow and the convergence of the K\"ahler Ricci flow. \\

 Suppose $\{(X^m, g(t)), t \in I \subset \R \}$ is a Ricci flow solution on a complete manifold $X^m$. Following the notations of~\cite{BWa2}, we define
\begin{align*}
  O_g(t)=\sup_X |R|_{g(t)}, \quad P_g(t)=\sup_X |Ric|_{g(t)}, \quad Q_g(t)=\sup_X |Rm|_{g(t)}.
\end{align*}
If the flow $g$ is obvious in the content, we will drop the subindex $g$.
We first state the improvement of the estimates in~\cite{BWa2}. \\

\begin{theoremin}[Riemannian curvature ratio estimate]
  There exists a constant $\epsilon_0=\epsilon_0(m)$ with the following properties.

  Suppose $K \geq 0$, $\left\{ (X, g(t)), -\frac{1}{8} \leq t \leq K \right\}$ is a Ricci flow solution on a complete manifold $X^m$,
  $Q(0)=1$, and $Q(t)\leq 2$ for every $t \in [-\frac{1}{8}, 0]$.  Suppose $T$ is the first time such that $Q(T)=2$, then there exists a
  point $x \in X$ and a nonzero vector $V \in T_x X$ such that
\begin{align}
  \left|  \log \frac{\langle V, V\rangle_{g(0)}}{\langle V, V\rangle_{g(T)}} \right| >\epsilon_0.
  \label{eqn:g24_1}
\end{align}
In particular, we have
  \begin{align}
     \int_0^T P(t)dt>\epsilon_0.
  \label{eqn:g22_2}
  \end{align}
  Consequently, we have
  \begin{align}
    Q(K)< 2^{\frac{\int_0^K P(t)dt}{\epsilon_0}+1}.
  \label{eqn:b23_10}
  \end{align}
  \label{thmin:c9_1}
\end{theoremin}

\begin{theoremin}[Ricci curvature estimate]
  Suppose $\left\{ (X, g(t)), -\frac{1}{8} \leq t \leq 0 \right\}$ is a Ricci flow solution satisfying  the following properties.
  \begin{itemize}
    \item $X$ is a complete manifold of dimension $m$.
    \item $|Rm|_{g(t)}(x) \leq 2$ whenever $x \in B_{g(0)}(x_0, 1), \; t \in \left[-\frac{1}{8}, 0 \right]$.
   \end{itemize}
  Then there exists a large constant $A=A(m)$ such that
  \begin{align}
    \sup_{B_{g(0)}\left(x_0, \frac{1}{2} \right) \times \left[-\frac{1}{16}, 0 \right]} |Ric|
    \leq A \norm{R}{L^{\infty}\left(B_{g(0)}(x_0, 1) \times \left[-\frac{1}{8}, 0 \right] \right)}^{\frac{1}{2}}.
    \label{eqn:ricrhalf}
  \end{align}
  \label{thmin:c9_2}
\end{theoremin}

The constants $\frac{1}{8}, \frac{1}{16}$ in Theorem~\ref{thmin:c9_1} and Theorem~\ref{thmin:c9_2} are chosen for the simplicity of applications.
They can be replaced by other constants $C, \frac{C}{2}$. Then of course $\epsilon_0=\epsilon_0(m, C)$, $A=A(m,C)$.
After we obtain these two universal dimensional constants $\epsilon_0$ and $A$, we can apply them
to improve the estimates of the constants in other theorems in~\cite{BWa2}. For simplicity, we will not list these improvements term by term.
Here we give some new applications.\\

\begin{corollaryin}[Ancient solution gap]
    Suppose $\left\{ (X^m, g(t)), -\infty <t \leq 0 \right\}$ is a non-flat, ancient Ricci flow solution such that
    $\displaystyle \limsup_{t \to -\infty} Q(t)|t|<\infty$.
    Then
    \begin{align}
       \limsup_{t \to -\infty} P(t)|t| \geq \frac{\epsilon_0}{\log 2}.
    \label{eqn:g25_2}
    \end{align}
    \label{clyin:g24_1}
\end{corollaryin}

\begin{corollaryin}[Shrinking soliton gaps]
 Suppose $(X^m, g)$ is complete, non-flat Riemannian manifold.
 If $(X^m, g)$ satisfies the shrinking Ricci soliton equation
 \begin{align}
   Ric + \mathcal{L}_V g -\frac{g}{2}=0
   \label{eqnin:ssoliton}
 \end{align}
 for some vector field $V$, $\displaystyle \sup_X |Rm|<\infty$, then
 \begin{align}
   A\sqrt{\sup_X |Rm|} \cdot \sqrt{\sup_X |R|}  \geq \sup_X |Ric| \geq \frac{\epsilon_0}{\log 2}.
 \label{eqnin:ssolitongap}
 \end{align}
  \label{clyin:ssolitongap}
\end{corollaryin}

Inspired by the proof of Theorem~\ref{thmin:c9_1} and Theorem~\ref{thmin:c9_2}, we also study the volume ratio for Ricci flow with bounded scalar curvature
and equivalent metrics (c.f. Proposition~\ref{prn:g23_1} and Proposition~\ref{prn:g22_1}). It turns out that
the volume ratio upper bound can be obtained even if the metric equivalence is absent. Of course, some other conditions are required as price.

\begin{theoremin}[volume ratio estimate]
  Suppose $\left\{ (X^m, g(t)), -1 \leq t \leq 0 \right\}$ is a Ricci flow solution on a closed manifold $X^m$ with the following properties.
  \begin{itemize}
    \item $\mu(g(t), \theta) \geq -B$ for every $t\in [-1, 0]$, $\theta \in (0, -t]$.
    \item $\displaystyle  -\sigma_{-} \leq R \leq \sigma_{+}$ on $X \times [-1, 0]$.
  \end{itemize}
  Then we have the volume ratio estimate
  \begin{align}
  \Vol_{g(0)} \left( B_{g(0)}(x_0, r) \right) \leq C(m, \sigma_{-}, \sigma_{+}, B) r^m,
   \label{eqn:g23_9}
  \end{align}
  for every $x_0 \in X$, $r \in (0, 1]$. Here $C(m, \sigma_{-}, \sigma_{+}, B) \leq (4\pi)^{\frac{m}{2}} e^{1+ \left(\sigma_{-} +\frac{1}{3}\sigma_{+} \right)+2\sqrt{B+\left(\sigma_{-}+\frac{1}{3}\sigma_{+}\right)+\frac{m}{2} \log \frac{4}{3}}}$.
  \label{thmin:vupper}
\end{theoremin}

\begin{corollaryin}
  Suppose $\left\{ (M^n, g(t)), 0 \leq t <\infty \right\}$ is a K\"ahler Ricci flow solution in the class $2\pi c_1(M)$ for some Fano manifold $M^n$.
  Then there is a constant $C=C(n, g(0))$ such that
  \begin{align*}
    Vol_{g(T)}(B_{g(T)}(x_0, r), r) \leq Cr^{2n}
  \end{align*}
  for every $x_0 \in M, r \in (0, 1], T \in [0, \infty)$.
  \label{clyin:krfupper}
\end{corollaryin}

\begin{remarkin}
Corollary~\ref{clyin:krfupper} generalize the volume ratio upper bound in~\cite{CWa1} to high dimension. However, the constant $C$ is not as precise as that of
$~\cite{CWa1}$ when $r$ is small. Corollary~\ref{clyin:krfupper} and part of Theorem~\ref{thmin:vupper} overlaps the main result of~\cite{Zhq3}.
\end{remarkin}

The organization of the paper is as follows.  We prove the curvature estimates (Theorem~\ref{thmin:c9_1} and Theorem~\ref{thmin:c9_2}) and their applications in section 2,
the volume estimate (Theorem~\ref{thmin:vupper}) and its applications in section 3.

\noindent {\bf Acknowledgment}
 The second author would like to thank Jeffrey Streets for helpful discussions during the preparation of this paper.

\section{Curvature estimates}

  Theorem~\ref{thmin:c9_1} is indicated by the following Proposition, which is more general.

  \begin{proposition}
    For every $\delta>0$, there exists a constant $\epsilon=\epsilon(m,\delta)$ with the following properties.

    Suppose that $\left\{ (X^m, g(t)), -\frac{1}{8} \leq t \leq 0\right\}$ is a Ricci flow solution on a complete manifold $X^m$ such that
    $Q_g(0)=1$ and $Q_g(t) \leq 2$ for every $-\frac{1}{8} \leq t \leq 0$.   If $\left\{ (X^m, h(t)), -\frac{1}{8} \leq t \leq 0 \right\}$ is another
    Ricci flow solution on $X$ such that $Q_h(t) \leq 2$ for every $-\frac{1}{8} \leq t \leq 0$ and $e^{-\epsilon} g(0) \leq h(0) \leq e^{\epsilon} g(0)$, then
    we have
    \begin{align*}
       \left| \log Q_h(0) \right|< \delta.
    \end{align*}
    \label{prn:g22_1}
  \end{proposition}

  \begin{proof}
     If this proposition were wrong, there should exist a constant $\delta_0$ and constants $\epsilon_i \to 0$ with corresponding Ricci flows
     $\left\{ (X_i^m, g_i(t)), -\frac{1}{8} \leq t \leq 0 \right\}$ and $\left\{ (X_i^m, h_i(t)), -\frac{1}{8} \leq t \leq 0 \right\}$ violating the statements.
  \begin{itemize}
    \item  $Q_{g_i}(t) \leq 2$ for every $-\frac{1}{8} \leq t \leq 0$, $Q_{g_i}(0)=1$.
    \item  $Q_{h_i}(t) \leq 2$ for every $-\frac{1}{8} \leq t \leq 0$, $|\log Q_{h_i}(0)| \geq \delta_0$.
    \item  $e^{-\epsilon_i} g_i(0) \leq h_i(0) \leq e^{\epsilon_i} g_i(0)$.
  \end{itemize}
  From the first and the second property and the definition that $\displaystyle Q(t)=\sup_X |Rm|(\cdot, t)$, we can choose points $z_i \in X_i$ such that
  \begin{align}
  \begin{cases}
   &\max\left\{ |Rm|_{h_i(0)}(z_i),  |Rm|_{g_i(0)}(z_i) \right\} \geq e^{-2 \delta_0}, \\
   &\left|\log \frac{|Rm|_{h_i(0)}(z_i)}{|Rm|_{g_i(0)}(z_i)} \right| \geq \frac{1}{2} \delta_0.
  \end{cases}
  \label{eqn:c17_pointchoice}
  \end{align}
  By the uniform curvature bound, the conjugate radii of $z_i$ under the metrics $g_i(0)$ are uniformly bounded from below by $2c_0=2c_0(m)$.
  Let $S$ be the standard ball in $\R^m$ with radius $c_0$, $\varphi_i$ be the exponential map from $\R^m=T_{z_i}X_i$ to $(X_i, g_i(0))$.
  Clearly, $\varphi_i$ is a local diffeomorphism from $S$ to $\varphi_i(S)$. We can use $\varphi_i$ to pull back the metrics:
  \begin{align}
    \tilde{g}_i(t) \triangleq \varphi_i^{*} g_i(t), \quad \tilde{h}_i(t) \triangleq \varphi_i^{*} h_i(t).
  \label{eqn:c18_tgi}
  \end{align}
  By Shi's local estimate along the Ricci flow, we have
  \begin{align*}
     |\nabla^k Rm|_{g_i(0)} \leq C(m, k),  \quad |\nabla^k Rm|_{h_i(0)} \leq C(m, k).
  \end{align*}
  It follows that
\begin{align}
  |\nabla^k Rm|_{\tilde{g}_i(0)}(p) \leq C(m, k), \quad |\nabla^k Rm|_{\tilde{h}_i(0)}(p) \leq C(m, k),\quad \forall \; p \in S, \; k \in \Z^{+} \cup \left\{ 0 \right\}.
\end{align}
Note that $inj_{\tilde{g}_i(0)}(o) \geq 2 c_0$. Combining this with the fact that $|Rm|_{\tilde{g_i}(0)} \leq 2$ implies that
$\Vol_{\tilde{g}_i(0)} \left( B_{\tilde{g}_i(0)}(p, \frac{c_0}{4}) \right)$ is uniformly bounded from below for every $p \in S$.
In light of the metric equivalence, we know that
$\Vol_{\tilde{h}_i(0)} \left( B_{\tilde{h}_i(0)}(p, \frac{c_0}{2}) \right)$ is uniformly bounded from below for every $p \in S$. Therefore, the estimate
of Cheeger-Gromov-Taylor (c.f.~\cite{CGT}) applies and we have
\begin{align*}
  inj_{\tilde{h}_i(0)}(p) \geq c_1(m)>0, \quad \forall \; p \in S.
\end{align*}
Now we collect the information we have
\begin{align}
\begin{cases}
  &\min\left\{ inj_{\tilde{g}_i(0)}(o), inj_{\tilde{h}_i(0)}(o) \right\} \geq c_1(m)>0, \\
  &\max\left\{ |\nabla^k Rm|_{\tilde{g}_i(0)}(p), |\nabla^k Rm|_{\tilde{h}_i(0)}(p) \right\} \leq C(k,m),
  \quad \forall k \in \Z^{+} \cup \left\{ 0 \right\}, \; p \in S, \\
  &\max\left\{ |Rm|_{\tilde{h}_i(0)}(o),  |Rm|_{\tilde{g}_i(0)}(o) \right\} \geq e^{-2\delta_0}, \\
  &\left|\log \frac{|Rm|_{\tilde{h}_i(0)}(o)}{|Rm|_{\tilde{g}_i(0)}(o)} \right| \geq \frac{1}{2} \delta_0.
\end{cases}
\label{eqn:c18_condition}
\end{align}
Therefore we can take smooth convergence in the Cheeger-Gromov sense:
\begin{align}
  &(S, o, \tilde{g}_i(0)) \stackrel{C.G.C^{\infty}}{\longrightarrow} (S',o',g'),  \label{eqn:c18_conv1}\\
  &(S, o, \tilde{h}_i(0)) \stackrel{C.G.C^{\infty}}{\longrightarrow} (S'', o'', g''). \label{eqn:c18_conv2}
\end{align}
Define $F_i$ to be the identity map:
\begin{align*}
 F_i:  & (S, o, \tilde{g}_i(0))    \mapsto (S, o, \tilde{h}_i(0)), \\
 & x \mapsto F_i(x)=x.
\end{align*}
In view of the third property, i.e., $e^{-\epsilon_i} g_i(0) \leq h_i(0) \leq e^{\epsilon_i} g_i(0)$,  we see that
\begin{align}
  \left| \log  \frac{d_{\tilde{h}_i(0)}(x, y)}{d_{\tilde{g}_i(0)}(F_i^{-1}(x), F_i^{-1}(y))} \right|
 =\left| \log  \frac{d_{\tilde{g}_i(0)}(x, y)}{d_{\tilde{h}_i(0)}(F_i(x), F_i(y))} \right|
 =\left| \log  \frac{d_{\tilde{g}_i(0)}(x, y)}{d_{\tilde{h}_i(0)}(x, y)} \right|
  \leq \epsilon_i \to 0, \quad \forall \; x, y \in S.
  \label{eqn:c18_1}
\end{align}
Therefore, we see that $F_i$ converges to an isometry map $F_{\infty}$:
\begin{align*}
 F_{\infty}:   (S', o', g') \to (S'', o'', g'');\quad  F_{\infty} (o')=o''.
\end{align*}
Since both $(S', o', g')$ and $(S'', o'', g'')$ are smooth geodesic balls, a theorem of Calabi-Hartman ~\cite{CH}
says that such an isometry $F_{\infty}$ must be smooth. Clearly, we have
\begin{align}
  |Rm|_{g'}(o')=|Rm|_{g''}(o'').
\label{eqn:c18_curv}
\end{align}
However, by the smooth convergence equations (\ref{eqn:c18_conv1}) and (\ref{eqn:c18_conv2}),  and the last two inequalities of
condition (\ref{eqn:c18_condition}), we obtain
\begin{align*}
  \max\left\{ |Rm|_{g'}(o'),  |Rm|_{g''}(o'') \right\} \geq e^{-2\delta_0}, \quad
  \left| \log \frac{|Rm|_{g'}(o')}{|Rm|_{g''}(o'')}\right| \geq \frac{1}{2} \delta_0.
\end{align*}
In particular, we obtain $|Rm|_{g'}(o') \neq |Rm|_{g''}(o'')$, which contradicts to equation (\ref{eqn:c18_curv}).
This contradiction establish the proof of Proposition~\ref{prn:g22_1}.
\end{proof}

\begin{proof}[Proof of Theorem~\ref{thmin:c9_1}]
After we obtain Proposition~\ref{prn:g22_1}, (\ref{eqn:g24_1}) and (\ref{eqn:g22_2}) follows from Proposition~\ref{prn:g22_1} trivially.
The proof of Theorem~\ref{thmin:c9_1} follows from the same argument as Theorem 3.1 of~\cite{BWa2}.
 We shall be sketchy here. Let $\epsilon_0=\epsilon(m, \log 2) =\epsilon_0(m)$ as in Proposition~\ref{prn:g22_1}, $t_i$ be the first time such that
$Q(t)=2^i$.  By the Gap property in Proposition~\ref{prn:g22_1}, we obtain that
\begin{align*}
  \int_{t_i}^{t_{i+1}} P(t) dt >\epsilon_0.
\end{align*}
Choose $N$ such that $\displaystyle 2^N \leq \sup_{t\in [0, K]} Q(t)<2^{N+1}$. It follows that
\begin{align*}
  \int_0^K P(t)dt \geq  \int_{0}^{t_N} P(t)dt>N\epsilon_0,  \quad \Rightarrow N < \frac{\int_0^K P(t)dt}{\epsilon_0},
  \quad \Rightarrow \sup_{t\in [0, K]} Q(t) \leq 2^{N+1} <2^{\frac{\int_0^K P(t)dt}{\epsilon_0}+1}.
\end{align*}
\end{proof}

In order to prove Theorem~\ref{thmin:c9_2}, we need to prove the following lemma first.

\begin{lemma}
  There is a constant $\delta_0=\delta_0(m)$ such that the following properties hold.

  Suppose $\left\{ (X^m, g(t)), -1 \leq t \leq 0 \right\}$ is a Ricci flow solution on the complete manifold $X^m$,
  $|Rm|_{g(t)} \leq \delta_0$ whenever $x \in B_{g(t)}(x_0, \delta_0^{-1})$, $t \in [-1, 0]$.  Then
  \begin{align}
    \sup_{B_{g(0)}(x_0, \frac{1}{2}) \times [-\frac{1}{2}, 0]} |Ric| \leq \delta_0^{-1} \norm{R}{L^{\infty}(B_{g(0)}(x_0, 10) \times [-1, 0])}^{\frac{1}{2}}.
    \label{eqn:c18_weakricci}
  \end{align}
  \label{lma:c18_weakricci}
\end{lemma}

\begin{proof}
Choose $\delta_0$ small enough. It will be determined later that how small $\delta_0$ is.

Identify $T_{x_0}X$ with $\R^m$. Let $\varphi$ be the exponential map from $T_{x_0}X$ to $X$ under the metric $g(0)$.
Since $|Rm|_{g(0)} \leq \delta_0 << \frac{1}{m^2}$, the conjugate radius of $g(0)$ at $x_0$ is far greater than $100$.
Therefore, $\varphi$ is a local diffeomorphism from $\Omega=B(o, 100)$, the geodesic ball of radius 100 on $T_{x_0} X$,
to $\varphi(\Omega) \subset  X$.   Define $\tilde{g}(t)= \varphi^{*} (g(t))$ for every $t \in [-1, 0]$.
Clearly, $inj_{\tilde{g}(0)}(x_0) \geq 100$. Moreover, by the estimates of Jacobi fields' lengths, we obtain that
\begin{align*}
  \Vol_{\tilde{g}(0)} \left( B_{\tilde{g}(0)}(x_0, r) \right) \geq \frac{1}{2} \omega(m) r^m, \quad \forall \; 0<r\leq 100,
\end{align*}
where $\omega(m)$ is the volume of the standard unit ball in $\R^m$.

It is not hard to see that $\left\{ (\Omega, \tilde{g}(t)), -1 \leq t \leq  0 \right\}$ is a Ricci flow solution.
It satisfies the following properties if we choose $\delta_0=\delta_0(m)$ very small.
\begin{itemize}
  \item  $|Rm|_{\tilde{g}(t)} \leq \delta_0 << \frac{1}{m^2}$ whenever
    $x \in B_{\tilde{g}(t)}(x_0, \delta_0^{-1})$, $t \in [-1, 0]$.
  \item $(B_{\tilde{g}(0)}(x_0, 100), \tilde{g}(t))$ has a uniform Sobolev constant $\sigma(m)$ for every $t \in [-1, 0]$.
  \item $inj_{\tilde{g}(t)}(x_0) \geq 3$ uniformly for every $t \in [-1, 0]$.
\end{itemize}
Therefore, the argument in Theorem 3.2 of~\cite{BWa2} applies. We have
\begin{align*}
  \sup_{B_{\tilde{g}(0)}(x_0, \frac{1}{2}) \times [-\frac{1}{2}, 0]} |Ric|
  \leq C \norm{R}{L^{\infty}(B_{\tilde{g}(0)}(x_0, 10) \times [-1, 0])}^{\frac{1}{2}}.
\end{align*}
Note that $\tilde{g}(0)$ is the metric lifted from $g(0)$. Therefore,
$\varphi (B_{\tilde{g}(0)}(x_0, r))=B_{g(0)}(x_0, r)$ whenever $r$ is less than the conjugate radius of $x_0$ under the metric
$g(0)$, which is far greater than $100$. It follows that
\begin{align*}
  \sup_{B_{g(0)}(x_0, \frac{1}{2}) \times [-\frac{1}{2}, 0]} |Ric|
  &=\sup_{B_{\tilde{g}(0)}(x_0, \frac{1}{2}) \times [-\frac{1}{2}, 0]} |Ric|\\
  &\leq C(m) \norm{R}{L^{\infty}(B_{\tilde{g}(0)}(x_0, 10) \times [-1,0])}^{\frac{1}{2}}\\
  &=C(m) \norm{R}{L^{\infty}(B_{g(0)}(x_0, 10) \times [-1,0])}^{\frac{1}{2}}\\
  &\leq \delta_0^{-1} \norm{R}{L^{\infty}(B_{g(0)}(x_0, 10) \times [-1,0])}^{\frac{1}{2}}.
\end{align*}
So we finish the proof of Lemma~\ref{lma:c18_weakricci}.
\end{proof}

\begin{proof}[Proof of Theorem~\ref{thmin:c9_2}]
By rescaling, we have the following equivalent property of Lemma~\ref{lma:c18_weakricci}.

Suppose $\left\{ (X^m, g(t)), -\delta_0^4 \leq t \leq 0 \right\}$ is a Ricci flow solution on the complete manifold $X^m$,
$|Rm|_{g(t)} \leq \delta_0^{-3}$ whenever $x \in B_{g(t)}(y_0, \delta_0)$, $t \in [-\delta_0^4, 0]$.  Then
  \begin{align}
    \sup_{B_{g(0)}(y_0, \frac{\delta_0}{2} ) \times [-\frac{\delta_0^4}{2}, 0]} |Ric|
    \leq \delta_0^{-3} \norm{R}{L^{\infty}(B_{g(0)}(y_0, 10\delta_0) \times [-\delta_0^4, 0])}^{\frac{1}{2}}.
    \label{eqn:c18_weakricci_2}
  \end{align}
  For every $y_0 \in B_{g(0)}(x_0, \frac{1}{2})$, $t \in [-\frac{1}{16}, 0]$, we have
  \begin{align*}
    |Ric|_{g(t)}(y_0) \leq \delta_0^{-3} \norm{R}{L^{\infty}(B_{g(t)}(y_0, 10\delta_0) \times [t-\delta_0^4, t])}^{\frac{1}{2}}
       \leq \delta_0^{-3} \norm{R}{L^{\infty}(B_{g(0)}(x_0, 1) \times [-\frac{1}{8}, 0])}^{\frac{1}{2}}.
  \end{align*}
  Let $A(m)=\delta_0^{-3}(m)$, we obtain
  \begin{align*}
    \sup_{B_{g(0)}(x_0, \frac{1}{2}) \times [-\frac{1}{16}, 0]} |Ric| \leq A \norm{R}{L^{\infty}(B_{g(0)}(x_0, 1)
     \times [-\frac{1}{8}, 0])}^{\frac{1}{2}}.
  \end{align*}
\end{proof}

Before we prove Corollary~\ref{clyin:g24_1}, let's first see the following gap inequality for $Q$ on any ancient solution:
 \begin{align}
   \liminf_{t \to -\infty} Q(t)|t| \geq  \frac{1}{8}.
 \label{eqn:g25_1}
 \end{align}
 Actually, by maximum principle, $Q(t)$ satisfies the equation $\displaystyle \D{}{t} Q \leq 8Q^2$.
 Choose $t_0$ such that $Q(t_0)>0$. Then for every $t \leq t_0 \leq  0$, ODE comparison implies
 \begin{align*}
  Q(t) \geq \frac{1}{8(t_0-t) +Q^{-1}(t_0)}.
 \end{align*}
 Then (\ref{eqn:g25_1}) follows trivially from the above inequality.

 If we regard $-\infty$ as a singular time.  Then (\ref{eqn:g25_1}) suggests that this singular time behaves similar to a finite time singularity when $Q(t)|t|$ is concerned.
 This observation inspires us to consider the gap of $P(t)|t|$ associated with the singular time $-\infty$.  Generally, such a gap may not exist. A trivial example is
 the Ricci-flat Ricci flow solution on a K3 surface, where $P(t)|t| \equiv 0$.     However, if we assume $Q(t)|t|$ is uniformly bounded, then such a gap does exist.

  \begin{proof}[Proof of Corollary~\ref{clyin:g24_1}]

  The proof is similar to the case of finite time singularity, c.f.~Theorem 1 of~\cite{BWa2}.

 If this theorem failed, we could find an ancient solution such that
 \begin{align*}
   \limsup_{t \to -\infty} Q(t)|t|<\infty,   \quad \limsup_{t \to -\infty} P(t)|t|<\eta \epsilon_0.
 \end{align*}
 for some positive number $\eta<\frac{1}{\log 2}$.
 Without loss of generality, we can assume
 \begin{align*}
   \frac{1}{16|t|} < Q(t) < \frac{C_0}{|t|},  \quad P(t) < \frac{\eta \epsilon_0}{|t|}, \quad \forall \; t \in (-\infty, -1].
 \end{align*}
 By a standard time selecting process, for every large positive integer number $L>0$, we can find a time  $t_L<-L$ such that
 \begin{align}
   Q(s)<2Q(t_L), \quad \forall \; s \in [t_L-\frac{1}{8Q(t_L)}, t_L].
 \label{eqn:I_19_1}
 \end{align}
 Applying a rescaling if necessary, Theorem~\ref{thmin:c9_1} and equation (\ref{eqn:I_19_1}) imply
 \begin{align*}
   Q(-1)<Q(t_L) 2^{\frac{\int_{t_L}^{-1} P(t) dt}{\epsilon_0}+1}.
 \end{align*}
 Therefore, we have
 \begin{align*}
   Q(t_L)> Q(-1) 2^{-\frac{\int_{t_L}^{-1} P(s) ds}{\epsilon_0}-1}>Q(-1) 2^{-\int_{t_L}^{-1} \frac{\eta}{|s|} ds -1}= Q(-1) 2^{- \eta \log |t_L| -1}
   =\frac{Q(-1)}{2} |t_L|^{-\eta \log 2}.
 \end{align*}
 It follows that
 \begin{align*}
  C_0 \geq  \limsup_{t \to -\infty} Q(t)|t| \geq  \lim_{L \to \infty}  \frac12 Q(-1) |t_L|^{1- \eta \log 2}=\infty,
 \end{align*}
 since $t_L \leq -L$ and $1-\eta \log 2 >0$. Contradiction!
 \end{proof}

 Corollary~\ref{clyin:ssolitongap} follows from Theorem~\ref{thmin:c9_2},
 Corollary~\ref{clyin:g24_1}, and the fact that every shrinking soliton with bounded curvature is a type-I ancient solution.

\section{Volume Estimates}

The following proposition is the motivation of Theorem~\ref{thmin:vupper}.

\begin{proposition}
Suppose  $\left\{ (X^m, g(t)), -1 \leq t \leq 0 \right\}$ is a Ricci flow solution on a closed manifold $X^m$, $x_0 \in X$.
Suppose also that
\begin{itemize}
  \item $\displaystyle -\sigma_{-}  \leq  R \leq \sigma_{+}$ on $B_{g(0)}(x_0, 1) \times [-1, 0]$.
  \item $g(x,t) \leq \Lambda g(x,0)$ for every $t \in [-1, 0], \; x \in B_{g(0)}(x_0,1)$.
\end{itemize}
Then we have
\begin{align}
  \Vol_{g(0)} \left(B_{g(0)}\left(x_0, r\right)\right) \leq C(m, \sigma_{-}, \sigma_{+}, \Lambda) r^m,
\label{eqn:g22_3}
\end{align}
for every $r \in (0, 1]$.  The constant $C(m, \sigma_{-}, \sigma_{+}, \Lambda)$ in equation (\ref{eqn:g22_3}) can be chosen as
$\displaystyle (4\pi)^{\frac{m}{2}} e^{\sigma_{-} +\frac{\sigma_{+}}{3} + \frac{\Lambda}{4}}$.
\label{prn:g23_1}
\end{proposition}

\begin{proof}
  Fix $q \in B_{g(0)}(x_0,1)$, let $\gamma$ be a shortest geodesic
  connecting $q, x_0$ under the metric $g(0)$ with unit speed.
  Let $a$ be the length of $\gamma$.
  Define a curve in space-time:
  \begin{align*}
    & \Gamma: [0, \bar{\tau}] \mapsto X \times [-1, 0], \\
    & \Gamma(s)= \left(\gamma \left(\sqrt{\frac{s}{\bar{\tau}}}a \right), s \right).
  \end{align*}
  Following the notation of~\cite{Pe1}, we can bound the reduced distance. 
  \begin{align*}
    \mathcal{L}(\Gamma)&= \int_0^{\bar{\tau}} \sqrt{\tau}  \left(R(\Gamma(\tau)) +|\dot{\Gamma}(\tau)|^2 \right) d\tau\\
    &\leq \sigma_{+} \int_0^{\bar{\tau}} \sqrt{\tau} d\tau + \int_0^{\bar{\tau}} \frac{a^2}{4\bar{\tau}\sqrt{\tau}} |\dot{\gamma}|^2 d\tau\\
    &\leq \sigma_{+} \int_0^{\bar{\tau}} \sqrt{\tau} d\tau + \Lambda \int_0^{\bar{\tau}} \frac{a^2}{4\bar{\tau}\sqrt{\tau}} d\tau\\
    &\leq \frac{2}{3} \sigma_{+} \bar{\tau}^{\frac{3}{2}} + \frac{1}{2} \Lambda a^2 \bar{\tau}^{-\frac12}.\\
    \Rightarrow   L(q, \bar{\tau}) &\leq \mathcal{L}(\Gamma)
    \leq \frac{2}{3} \sigma_{+} \bar{\tau}^{\frac{3}{2}} + \frac{1}{2} \Lambda a^2 \bar{\tau}^{-\frac12}.\\
    \Rightarrow   l(q, \bar{\tau}) &\leq \frac{1}{2\sqrt{\bar{\tau}}} L(q, \bar{\tau}) \leq \frac13 \sigma_{+} \bar{\tau} + \frac14 \Lambda a^2 \bar{\tau}^{-1}.
  \end{align*}
  Let $\Omega = B_{g(0)}(x_0, \sqrt{\bar{\tau}})$. Clearly, $l(q, \bar{\tau}) \leq \frac{1}{3}\sigma_{+} \bar{\tau} + \frac{1}{4} \Lambda$ in $\Omega$.
  By the monotonicity of reduced volume, we have
  \begin{align*}
    (4\pi)^{\frac{m}{2}} \geq  \tilde{V}(0) \geq   \tilde{V}(\bar{\tau})
    =\bar{\tau}^{-\frac{m}{2}} \int_X e^{-l(q, \bar{\tau})} d\mu_q
    \geq \bar{\tau}^{-\frac{m}{2}}   e^{-\frac13 \sigma_{+} \bar{\tau}-\frac14 \Lambda } \Vol_{g(-\bar{\tau})}(\Omega),
 \end{align*}
 which implies
 $\displaystyle \Vol_{g(-\bar{\tau})}(\Omega) \leq  e^{\frac13 \sigma_{+} \bar{\tau}+\frac14 \Lambda}  (4\pi \bar{\tau})^{\frac{m}{2}}$.
 On the other hand, evolution of volume under Ricci flow yields
  \begin{align}
    &\quad \D{}{t} \Vol_{g(t)}(\Omega) = \int_{\Omega} (-R)  \leq \sigma_{-} \Vol_{g(t)}(\Omega), \nonumber \\
    &\Rightarrow   \Vol_{g(0)}(\Omega) \leq \Vol_{g(-\bar{\tau})}(\Omega) e^{\sigma_{-}\bar{\tau}}
    \leq e^{ (\frac13 \sigma_{+} + \sigma_{-}) \bar{\tau}+\frac14 \Lambda}  (4\pi \bar{\tau})^{\frac{m}{2}}. \label{eqn:g23_1}
  \end{align}
  Let $r=\sqrt{\bar{\tau}}$. Note that $\bar{\tau}=r^2 \leq 1$.   Proposition~\ref{prn:g23_1} follows from inequality (\ref{eqn:g23_1}).
 \end{proof}

 The idea of the previous proposition is quite simple: try to use reduced volume to estimate volume.  If we assume scalar curvature is uniformly bounded, it is not important
 to differentiate the volume at different time slices.  Therefore, the proof of Proposition~\ref{prn:g23_1} indicates the following observation.
 It is sufficient to prove the volume upper bound by finding a nonnegative function $w_r$ with the following properties.
 \begin{itemize}
   \item  $\int_{B_{g(0)}(x_0, r)} w_r d\mu_{g(0)} \leq C_1$.
   \item  $w_r \geq C_2^{-1} r^{-m}$ on the geodesic ball $B_{g(0)}(x_0, r)$.
 \end{itemize}
 In Proposition~\ref{prn:g23_1}, $e^{-l} \bar{\tau}^{-\frac{m}{2}}$, which is the Jacobi of reduced volume, plays the role of $w_r$ (without considering the difference of volume elements
 at different time slices).  Generally, if such $w_r$ exists, we have
   \begin{align}
     \Vol_{g(0)} \left(B_{g(0)}(x_0, r) \right) C_2^{-1} r^{-m} \leq C_1,
     \Rightarrow  \; \Vol_{g(0)} \left(B_{g(0)}(x_0, r) \right) \leq C_1 C_2 r^{m}.
   \label{eqn:g23_7}
   \end{align}
   Clearly, $e^{-l} \bar{\tau}^{-\frac{m}{2}}$ approaches the fundamental solution at $\bar{\tau} \to 0$, for both the operator $\square=\D{}{t}-\Delta$ and
   its conjugate operator $\square^*=\D{}{\tau} -\Delta + R$.  Since $R\bar{\tau} \to 0$, the extra term $R$ is not important when we study the limit behavior.
   Therefore, there are two candidates for $w_r$: the fundamental solution of $\square$ and the fundamental solution of $\square^*$.
   Inspired by the work of~\cite{CaZh}, we found that the fundamental solution of $\square$ works.
   The function $w_r$ can be chosen as the fundamental solution of $\square w=0$ based at the point $(x_0, -r^2)$.

\begin{proof}[Proof of Theorem~\ref{thmin:vupper}]
From the previous discussion, it suffices to prove the following properties.
  \begin{itemize}
    \item
    \begin{align}
     \int_{B_{g(0)}(x_0, r)} w(\cdot, 0) < \int_X w(\cdot, 0) \leq e^{\sigma_{-} r^2}.\label{eqn:g23_6}
    \end{align}
    \item   In the ball $B_{g(0)}(x_0, r)$, we have
    \begin{align}
      w(\cdot, 0) \geq (4\pi)^{-\frac{m}{2}} \cdot e^{-1-\frac{1}{3}\sigma_{+}r^2-2\sqrt{B+\left(\sigma_{-}+\frac{1}{3}\sigma_{+}\right) r^2+\frac{m}{2} \log \frac{4}{3}}} \cdot r^{-m}.
    \label{eqn:g23_8}
    \end{align}
  \end{itemize}
  Here $w$ is the fundamental solution of $\square w=0$ satisfying $\displaystyle \lim_{t \to -r^2} w(x, t) = \delta_{x_0}$.

  The first property follows from direct computation:
  \begin{align*}
  \qquad \frac{d}{dt} \left( \int w \right)=\int (\Delta w - Rw)= \int -Rw \leq \sigma_{-} \int w,
  \Rightarrow \left. \left(\int w\right)\right|_0 \leq  \left.\left(\int w\right)\right|_{-r^2} e^{\sigma_{-} r^2} =e^{\sigma_{-} r^2}.
 \end{align*}

The second inequality is more complicated. We first need to show that $w(x_0, 0)$ is comparable to $r^{-m}$, then we apply a gradient estimate to show that
$w(x, 0)$ is comparable to $r^{-m}$ in the geodesic ball $B_{g(0)}(x_0, r)$. Most of these estimates are available in the paper~\cite{CaZh}.
We only need some modifications to match our condition. For the completeness, we give a sketchy proof here.

\begin{claim}
  For every $t \in (-r^2, 0]$, we have
  \begin{align}
    \norm{w(t)}{L^{\infty}} \leq \left(4\pi \left(t+r^2 \right) \right)^{-\frac{m}{2}} \cdot e^{B+\sigma_{-}(t+r^2)}.
  \label{eqn:g23_5}
  \end{align}
  In particular, at $t=0$, we have
  \begin{align}
    \norm{w(0)}{L^{\infty}} \leq (4\pi)^{-\frac{m}{2}} \cdot e^{B+ \sigma_{-}r^2} \cdot r^{-m}.
  \label{eqn:infby1_g23}
  \end{align}
\label{clm:g23_1}
\end{claim}

  Let $p(t)=\frac{r^2}{-t}, \; t \in [-r^2,0]$. Clearly $p(-r^2)=1, p(0)=\infty$.
  Direct computation shows
  \begin{align}
    \left(\log \norm{w}{L^p} \right)' &= -p^{-2}p' \log \left( \int w^p \right) + p^{-1}\left( \int w^p \right)^{-1} \left( \int w^p \right)', \nonumber \\
     &=-p^{-2}p' \log \left( \int w^p \right) + \int w^p \left( p' \log w + pw^{-1}w' -R \right).
     \label{eqn:k25_1}
  \end{align}
  Let $v$ be the normalization of $w$: $v=A^{-1}w^{\frac{p}{2}}$ where $A=\left( \int w^p \right)^{\frac{1}{2}}$.
  Then equation (\ref{eqn:k25_1}) reads
  \begin{align}
    \left(\log \norm{w}{L^p} \right)' &=p^{-2}p' \int v^2 \log v^2 + \int v^2 w^{-1}w'   -p^{-1}\int Rv^2, \nonumber \\
     &=p^{-2}p' \int v^2 \log v^2 - 4p^{-2}(p-1)\int \snorm{\nabla v}{}^2-p^{-1}\int Rv^2, \label{eqn:g23_2}
  \end{align}
  where we use the fact $w'=\Delta w$ in the deduction of (\ref{eqn:g23_2}).
  By the definition of $p(t)$, we have
  \begin{align*}
     p^{-2}p'=r^{-2},  \quad p^{-2}(p-1)= \frac{(t+r^2)|t|}{r^4}.
  \end{align*}
  Equation (\ref{eqn:g23_2}) can be simplified as
  \begin{align}
    r^2 \left(\log \norm{w}{L^p} \right)' &=\int v^2 \log v^2 -\frac{|t|(t+r^2)}{r^2} \int \left(4\snorm{\nabla v}{}^2+Rv^2\right) -\frac{t^2}{r^2} \int Rv^2 \nonumber\\
    &\leq B -m -\frac{m}{2} \log \left(\frac{4\pi|t|(t+r^2)}{r^2} \right) +\frac{t^2}{r^2} \sigma_{-}.
    \label{eqn:k25_5}
  \end{align}
  In the last step, we used  the following proposition which follows from the definition of Perelman's $\mu$-functional and our condition:
  For every $\theta \in (0,|t|]$ and every smooth function $\Psi$ with $\int_X \Psi^2=1$, we have
  \begin{align*}
    \int_X \Psi^2 \log \Psi^2 - \theta \left(R\Psi^2 + 4\snorm{\nabla \Psi}{}^2 \right) \leq B-m-\frac{m}{2} \log (4\pi \theta).
  \end{align*}
  Integrate inequality (\ref{eqn:k25_5}) on both sides, we have
  \begin{align*}
    r^2 \log \frac{\norm{w(0)}{L^{\infty}}}{\norm{w(-r^2)}{L^1}}
    &\leq (B-\frac{m}{2} \log (4\pi r^2))r^2 + r^4 \sigma_{-},
 \end{align*}
  which implies
  \begin{align*}
    \norm{w(0)}{L^{\infty}} \leq \norm{w(-r^2)}{L^1} (4\pi r^2)^{-\frac{m}{2}}  e^{B+ \sigma_{-}r^2}
    =(4\pi r^2)^{-\frac{m}{2}}  e^{B+ \sigma_{-}r^2}.
  \end{align*}
  This finishes the proof of (\ref{eqn:infby1_g23}).  Replacing $r^2$ by $r^2+t$, we obtain (\ref{eqn:g23_5}).
  Therefore, we finish the proof of Claim~\ref{clm:g23_1}. \\

  \begin{claim}
    \begin{align}
     w(x_0, 0) \geq  (4\pi)^{-\frac{m}{2}} \cdot e^{-\frac{1}{3} \sigma_{+} r^2} \cdot r^{-m}.
    \label{eqn:g23_3}
   \end{align}
  \label{clm:g23_2}
\end{claim}

  Let $u$ be the fundamental solution of $\left( \D{}{\tau} - \Delta + R \right) u=0$ such that $\displaystyle \lim_{t \to 0} u =\delta_{x_0}$, where $\tau=-t$.
  Since  $\int_X wu$ is a constant, it is easy to see that $w(x_0, 0)=u(x_0, r^2)$.  Therefore, it suffices to develop the lower bound of $u(x_0, r^2)$.

  Let $f=-\log u -\frac{m}{2} \log (4\pi \tau)$. Applying Perelman's Harnack inequality (Corollary 9.4 of ~\cite{Pe1}) for the fixed curve $\gamma(\tau) \equiv x_0$,
  we obtain
   \begin{align}
     \D{}{\tau}\left( \sqrt{\tau} f(x_0, \tau) \right) \leq \frac{\sqrt{\tau}}{2} R(x_0, -\tau) \leq \frac{\sqrt{\tau}}{2} \sigma_{+}, \label{eqn:g23_4}
  \end{align}
  Note that $f(x_0, \tau) \sim \frac{d_{g(0)}(x_0, x_0)}{\tau} =0$ as $t \to 0$. Integrating (\ref{eqn:g23_4}) yields
  $\displaystyle  f(x_0, r^2) \leq  \frac{1}{3} \sigma_{+} r^2$, which in turn implies (\ref{eqn:g23_3}) by the definition of $f$ and the fact that $u(x_0, r^2)=w(x_0, 0)$.
  So we finish the proof of Claim~\ref{clm:g23_2}.\\

  By gradient estimate of $w$, c.f. Theorem 3.3 of~\cite{Zhq2}, or Theorem 5.1 of~\cite{CaHa}, one has
  \begin{align}
  w(x, 0) \geq w(x_0, 0) e^{\frac{-d^2(x,x_0,0)}{4|t_0|} -\frac{d(x,x_0,0)}{\sqrt{|t_0|}} \cdot \sqrt{\log \frac{\sup_{X \times [t_0,0]} w}{w(x_0,0)}}},
  \label{eqn:g22_4}
  \end{align}
  for every $t_0 \in (-r^2, 0]$.  Let $t_0=-\frac{r^2}{4}$. When $x \in B_{g(0)}(x_0, r)$, by (\ref{eqn:g23_5}), (\ref{eqn:g23_3}), and (\ref{eqn:g22_4}), we obtain
  \begin{align*}
    w(x,0) &\geq w(x_0, 0) e^{-1-2\sqrt{B+\left(\sigma_{-}+\frac{1}{3}\sigma_{+}\right) r^2+\frac{m}{2} \log \frac{4}{3}}}\\
    &\geq (4\pi)^{-\frac{m}{2}} \cdot e^{-1-\frac{1}{3}\sigma_{+}r^2-2\sqrt{B+\left(\sigma_{-}+\frac{1}{3}\sigma_{+}\right) r^2+\frac{m}{2} \log \frac{4}{3}}} \cdot r^{-m}.
  \end{align*}
It follows that
\begin{align*}
  \Vol_{g(0)} \left( B_{g(0)}(x_0, r) \right)
  \leq (4\pi)^{\frac{m}{2}}
    \cdot  e^{1+ \left(\sigma_{-} +\frac{1}{3}\sigma_{+} \right)r^2+2\sqrt{B+\left(\sigma_{-}+\frac{1}{3}\sigma_{+}\right) r^2+\frac{m}{2} \log \frac{4}{3}}}
    \cdot r^m
  \leq C(m, \sigma_{-}, \sigma_{+}, B) r^m.
\end{align*}
\end{proof}

\begin{remark}
  Proposition~\ref{prn:g23_1} is purely local. Unfortunately, the metric bound condition is hard to obtain.   Theorem~\ref{thmin:vupper} is based on natural condition.
  However, we pay price by sacrificing the local property.  Both the upper bounds in Proposition~\ref{prn:g23_1} and Theorem~\ref{thmin:vupper} are not precise even if
  the Ricci flow is flat.  The requirement that $X$ is closed can be dropped if we have good bounds of decaying speed of fundamental solutions of $\square$ and $\square^*$
  at space infinity.
  \label{rmk:g24_1}
\end{remark}

\begin{proof}[Proof of Corollary~\ref{clyin:krfupper}]
Note that $g(t)$ satisfies the normalized Ricci flow equation $\displaystyle \D{}{t} g = -Ric + g$.
Without loss of generality, we assume $T \geq \log \frac{3}{2}$.
Define $\displaystyle \tilde{g}(s) =\left( 1- \frac{s}{2} \right) g\left(T+\log \left(1- \frac{s}{2} \right)^{-1} \right)$, $2\left(1-e^T \right) \leq s <2$.
Clearly, $\tilde{g}$ is an unnormalized Ricci flow solution with $\tilde{g}(0)=g(T)$.

Suppose $\sigma$ is the uniform bound of scalar curvature along $g(t)$, then $|\tilde{R}| \leq \sigma$ on $M \times [-1, 0]$. Moreover,
by the monotonicity and scaling property of the $\mu$-functional along the Ricci flow, we have
\begin{align*}
  \mu\left( \tilde{g}(t), \theta \right) &\geq \mu \left( \tilde{g} \left(2\left(1-e^T\right) \right), 2\left(e^T-1\right)+t+\theta \right)=\mu \left( e^T g(0), 2\left(e^T-1\right)+t+\theta \right)\\
    &=\mu \left( g(0), 2 + e^{-T} \left( t+\theta -2\right) \right),
\end{align*}
for every $t \in [-1, 0]$ and $\theta \in [0, -t]$. Note that
\begin{align*}
 & 0 \leq 2-3e^{-T} \leq  2+e^{-T} \left( t+\theta -2 \right) \leq 2-2e^{-T}<2, \\
 &\lim_{T \to \infty}  2+e^{-T} \left( t+\theta -2 \right)=2.
\end{align*}
Define $\displaystyle -B=\inf_{0<\tau \leq 2} \mu(g(0), \tau)$.  Actually, $-B$ can be chosen as close to $\mu(g(0), 2)$ as possible whenever
$T$ is very large.
Therefore, we have the following conditions:
\begin{itemize}
  \item  $-\sigma \leq \tilde{R} \leq \sigma$ on $M \times [-1, 0]$.
  \item  $\mu(\tilde{g}(t), \theta) \geq -B$ for every $t \in [-1, 0]$ and $\theta \in [0, -t]$.
\end{itemize}
Therefore, Theorem~\ref{thmin:vupper} yields the desired volume ratio upper bound under metric $\tilde{g}(0)=g(T)$.
\end{proof}

Another application of Theorem~\ref{thmin:vupper} is to study the extension of Ricci flow whenever scalar curvature is uniformly bounded.
Suppose $\left\{ (X^m, g(t)), -1 \leq t <0 \right\}$ is a Ricci flow solution on a closed manifold,  whose scalar curvature is uniformly bounded by $\sigma$.
Clearly, for every $t \in [-1, 0)$, $\theta \in (0, -t]$,  we have $\mu (g(t), \theta) \geq \mu\left( g(-1), 1+t+\theta \right)>-B$ for some constant $B$.
Therefore, since $\displaystyle \lim_{t \to 0} \Vol_{g(t)}(X) \geq e^{-\sigma} \Vol_{g(-1)}(X)>0$, Theorem~\ref{thmin:vupper}
can be applied to show that $\displaystyle \lim_{t \to 0} \diam_{g(t)}(X)>0$, which is a positive evidence of the problem
which claims that the Ricci flow can be extended whenever scalar curvature is uniformly bounded. This will be discussed in the future.

\vspace{1in}

 Xiuxiong Chen,  Department of Mathematics, University of
 Wisconsin-Madison, Madison, WI 53706, USA; xiu@math.wisc.edu\\

 Bing  Wang,   Department of Mathematics, Princeton University,
  Princeton, NJ 08544, USA; bingw@math.princeton.edu\\

\end{document}